\documentclass[12pt,a4paper]{article}

\usepackage{amsmath}
\usepackage{amsfonts}
\usepackage{amssymb}
\usepackage{amsthm}

\newtheorem{theorem}{Theorem}
\newtheorem*{theorem*}{Theorem}
\newtheorem*{corol*}{Corollary}

\author{Sagar Shrivastava \& Vaibhav Pandey}
\date{}
\title{Ring of Real Analytic Functions on $[0,1]$}

\begin{document}
\maketitle

\begin{abstract}
We consider the ring of real analytic functions defined on $[0,1]$, i.e.
$$C^{\omega}[0,1] =\lbrace f :[0,1] \longrightarrow \mathbb{R} | f \text{ is analytic on } [0,1]\rbrace$$
In this article, we explore the nature of ideals in this ring. It is well known that the ring $C[0,1]$ of real valued continuous functions on $[0,1]$ has precisely the following maximal ideals:
$$\text{For } \gamma \in [0,1], M_{\gamma} := \lbrace f \in C[0,1]  | f(\gamma) =0\rbrace$$
It has been proved that each such $M_{\gamma}$ is infinitely generated, in-fact uncountably generated \cite{sury2011uncountably}. Observe that $C^{\omega}[0,1]$ is a subring of $C[0,1]$ 

We prove that for any $\gamma$ in $[0,1]$, the contraction $M^{\omega}_{\gamma}$ of $M_{\gamma}$ under the natural inclusion of $C^{\omega}[0,1]$ in $C[0,1]$ is again a maximal ideal (of $C^{\omega}[0,1]$ ), and these are precisely all the maximal ideals of $C^{\omega}[0,1]$. Next we prove that each $M^{\omega}_{\gamma}$ is principal (though $M_{\gamma}$ is uncountably generated). Surprisingly,  this forces all the ideals of the ring $C^{\omega}[0,1]$ to be singly generated, i.e. $C^{\omega}[0,1]$ is a PID.
\end{abstract}
\newpage

To begin with, we observe that the maximal ideals of $C^{\omega}[0,1]$ are precisely $M^{\omega}_{\gamma} (\gamma \in [0,1])$.

Firstly, $M^{\omega}_{\gamma}$ is the kernel of the following surjective homomorphism:
\begin{align*}
\Phi_{\gamma} : C^{\omega}[0,1]&\longrightarrow \mathbb{R}\\
f &\mapsto f(\gamma)
\end{align*}
Hence $C^{\omega}[0,1] /M^{\omega}_{\gamma} \cong \mathbb{R}$. Therefore $M^{\omega}_{\gamma}$ is maximal.

Further, the fact that any maximal ideal of $C^{\omega}[0,1]$ is of the form $M^{\omega}_{\gamma}$  for some $\gamma \in [0,1]$ is merely a consequence of compactness of $[0,1]$. The proof for this is precisely the proof of the fact that any maximal ideal of $C[0,1]$ is of the form $M_{\gamma}$.

\begin{theorem} \label{th1}
$M^{\omega}_{\gamma}$ is a principal ideal generated by $(x-\gamma)$
\end{theorem}

\begin{proof}
Before we come to the main proof, note that
$$M^{\omega}_{\gamma} = \lbrace f(x) -f(\gamma) | f \in C^{\omega}[0,1] \rbrace$$
Let $f$ be any function in $C^{\omega}[0,1]$, we define
$$g(x) := \frac{f(x) -f(\gamma)}{x-\gamma}$$
We claim that $g$ is also analytic on $[0,1]$. For a moment, assume the claim. Then we have $f(x) -f(\gamma) =  (x-\gamma) g(x)$, and hence any arbitrary element $f$ of $M^{\omega}_{\gamma}$ is contained in the ideal generated by $(x-\gamma)$, hence $M^{\omega}_{\gamma}$ is a principal ideal.

So now let us prove the claim. We prove our claim in two parts. First we show that there exists a convergent Taylor series for $g$ in a neighborhood of $\gamma$ and then we show that in any neighborhood not containing $\gamma$, it is analytic, and hence analytic on the whole of $[0,1]$.

Since $f$ is analytic, we can get a Taylor series expansion of $f$ at $\gamma$. Then for $x$ in a neighborhood of $\gamma$, we have-
\begin{align*}
f(x) &= \sum_{n=0}^{\infty} f^{(n)}(\gamma) \frac{(x- \gamma)^n}{n!}\\
\Rightarrow f(x) -f(\gamma) &= (x- \gamma)  \sum_{n=1}^{\infty} f^{(n)}(\gamma) \frac{(x- \gamma)^{n-1}}{n!}\\
\Rightarrow \frac{f(x) -f(\gamma)}{x- \gamma} &=   \sum_{n=1}^{\infty} f^{(n)}(\gamma) \frac{(x- \gamma)^{n-1}}{n!} = \sum_{n=0}^{\infty} \frac{f^{(n+1)}(\gamma)}{n+1} \frac{(x- \gamma)^{n}}{n!}
\end{align*}

The LHS of the above equation is $g(x)$ by definition, and we can say that the RHS is a Taylor series expansion of $g(x)$ where $g^{(n)}(\gamma) = \frac{f^{(n+1)}(\gamma)}{n+1} $. We claim that the RHS is convergent in a neighborhood of $\gamma$.
Since $f$ is analytic, its derivative ($f'$) is also analytic on a neighborhood around $\gamma$, and we get the following Taylor series expansion at $\gamma$ given by
$$f'(x) = \sum_{n=0}^{\infty}f^{(n+1)}(\gamma) \frac{(x- \gamma)^{n}}{n!}$$

Clearly, the coefficients of $\frac{(x- \gamma)^{n}}{n!}$ in the expansion of $f'$ and $g$ satisfy the following inequality in the intersection of the neighborhoods of $\gamma$ where $f$ and $f'$ are expressed in their Taylor series-
$$|g^{(n)}(\gamma)| = \frac{|f^{(n+1)}(\gamma)|}{n+1} \leq |f^{(n+1)}(\gamma)| $$
Since the Taylor series of $f'$ converges in this neighborhood, it forces the Taylor series of $g$ to converge in this neighborhood of $\gamma$. Hence $g$ is analytic in some neighborhood of $\gamma$ (call that neighborhood N, where $N=[0,\varepsilon)$ if $\gamma=0$, $(1-\delta,1]$ if $\gamma=1$, $(\gamma-\delta, \gamma+\varepsilon)$ otherwise, for $\varepsilon, \delta >0$)

For the other part of the claim, we show that $g$ is analytic on the set $N'$ ($N' = (\varepsilon/2,1]$ if $\gamma =0$, $N' = [0,1-\delta/2)$ if $\gamma =1$, $N' =[0,\gamma-\delta/2) \cup (\gamma +\varepsilon/2,1]$ otherwise, where $\varepsilon, \delta >0$ are the same as used for $N$)

Now $(x- \gamma)$ is analytic on $N'$, and $x- \gamma \neq 0$ in $N'$.\\
$\therefore \frac{1}{(x- \gamma)}$ is analytic on $N'$, and $f(x) -f(\gamma)$ is also analytic on $N'$ ( as it is analytic on $[0,1]$). So,

$$g(x) = (f(x) -f(\gamma)) \frac{1}{(x- \gamma)} = \frac{f(x) -f(\gamma)}{x-\gamma} $$
As $g$ is a product of two analytic functions on $N'$, it is also an analytic function on $N'$.

Hence $g$ is an analytic function on $N \cup N' = [0,1]$
So we get that $M^{\omega}_{\gamma}$ is a principal ideal.
\end{proof}

\begin{theorem}
$C^{\omega}[0,1]$ is an integral domain.
\end{theorem}

\begin{proof}
Let $f,g \in C^{\omega}[0,1]$ be such that $f.g\equiv 0$. We will use the fact that an analytic function is if and only if all the coefficients in the Taylor series expansion are zero. So we expand the functions $f,g$ along $\frac{1}{2}$, without loss of generality assume that $f \not\equiv 0$, hence we will show that $g \equiv 0$. But this will be the same as proving that power series ring over an integral domain is again an integral domain, which proves that $C^{\omega}[0,1]$ is an integral domain.
\end{proof}

\begin{corol*}
$C^{\omega}[0,1]$ is a PID.
\end{corol*}
\begin{proof}
Let $I$ be a non-zero proper ideal of $C^{\omega}[0,1]$. Then it must be contained in some maximal ideal. Then by \eqref{th1}, $I \subseteq\langle x- \gamma_1\rangle$ for some $\gamma_1 \in [0,1]$.

Hence $I = \langle x- \gamma_1\rangle I'$, for some ideal $I'$ of $C^{\omega}[0,1]$. If $I' = C^{\omega}[0,1]$, $I = \langle x- \gamma_1\rangle = M_{\gamma_1}^{\omega}$ and we are done.

If not, then $I'$ is a proper ideal of $C^{\omega}[0,1]$ and $I' \subseteq \langle x- \gamma_2\rangle$, so that $I = \langle x- \gamma_1\rangle\langle x- \gamma_2\rangle I''$ for some ideal $I''$ and we continue this process.

We claim that this process will stop after finitely many steps. If not, then choose any non zero function $f \in I$, then $f(\gamma_i) =0$ $\forall i \in \mathbb{N}$. The set $\left\lbrace \gamma_i \right\rbrace_{i \in \mathbb{N}}$ is an infinite set in the compact set $[0,1] $, hence by Bolzano-Weistrass theorem, this set has a limit point in $[0,1]$. Also if the zero set of an analytic function has a limit point then the function itself is zero in the connected component containing the limit point. Therefore $f$ is zero in $[0,1]$, which contradicts our assumption that $f$ is nonzero. Hence the process stops after finitely many steps.

Note that the set $\left\lbrace \gamma_i \right\rbrace_{i \in \mathbb{N}}$ may have repetitions, but it does not affect our argument.

Hence $I = \langle x- \gamma_1\rangle \langle x- \gamma_2\rangle \ldots \langle x- \gamma_n\rangle = \langle \prod_{i=1}^{n}(x- \gamma_i)\rangle$
Hence $I$ is principal, and $C^{\omega}[0,1]$ is a PID.
\end{proof}

The way one shows the existence of non-maximal prime ideals in the ring $C[0,1]$ is as follows: Any maximal element of the set of ideals disjoint from the multiplicatively closed set of non-zero polynomials partially ordered by inclusion can be checked to be a non-maximal prime ideal of $C[0,1]$. Let $P$ be one such maximal element. Further, any prime ideal of $C[0,1]$ is contained in a 'unique' maximal ideal of $C[0,1]$ [Pan16]. (The above outline of the proof through Zorn's lemma is perhaps the only way to show the existence of non-maximal prime ideals in $C[0,1]$.)

\begin{corol*}
The non maximal prime ideal $P$ as mentioned above contains no nonzero analytic functions.
\end{corol*}

\begin{proof}
Consider the contraction of $P$ in the ring $C^{\omega}[0,1]$. It has to be a prime ideal of $C^{\omega}[0,1]$. Let $M_{\gamma}$ be the unique maximal ideal of $C[0,1]$ such that $P \subset M_{\gamma}$. Then, the contraction of $P$ in the ring $C^{\omega}[0,1]$ is either $0$ or $M_{\gamma}^{\omega} = \langle x - \gamma \rangle$. But $x-\gamma$ is not in $P$, therefore $P$ contracts to the zero ideal in $C^{\omega}[0,1]$.  
\end{proof}

The fact that $C^{\omega}[0,1]$ is a PID while $C[0,1]$ has infinitely many uncountably generated ideals gives us an algebraic evidence to our intuition that real analytic functions are sparse as compared to real continuous functions (on a compact set).

\nocite{*}

\clearpage
\bibliographystyle{alpha}

\end{document}